\documentclass[11pt]{amsart}
\usepackage[active]{srcltx}
\usepackage[utf8]{inputenc}
\usepackage{amsmath,mathtools}
\usepackage{stmaryrd}
\usepackage{MnSymbol}
\usepackage{hyperref}
\usepackage{tikz}
\usetikzlibrary{arrows,automata}
\usepackage{enumitem}
\usepackage{mathrsfs}
\usepackage{pst-plot}
\usepackage{auto-pst-pdf}
\usepackage[all]{xy}
\usepackage{stackrel}
\usepackage{listings}
\lstdefinelanguage{GAP}{%
  morekeywords={%
    Assert,Info,IsBound,QUIT,%
    TryNextMethod,Unbind,and,break,%
    continue,do,elif,%
    else,end,false,fi,for,%
    function,if,in,local,%
    mod,not,od,or,%
    quit,rec,repeat,return,%
    then,true,until,while%
  },%
  sensitive,%
  morecomment=[l]\#,%
  morestring=[b]",%
  morestring=[b]',%
}[keywords,comments,strings]
\usepackage[T1]{fontenc}
\usepackage{xcolor}
\usepackage{enumitem}
\newtheorem{thm}{Theorem}[section]
\newtheorem{cor}[thm]{Corollary}

\newtheorem{lem}[thm]{Lemma}
\newtheorem{prop}[thm]{Proposition}

\newcommand{\abs}[1]{\left\vert#1\right\vert}

\def\pv#1{\ensuremath{\mathsf{#1}}}

\newcommand{\DDet}{\mathop{\mathrm{Det}}\nolimits}

\usepackage{caption}

\makeatletter
\newcommand{\mylabel}[2]{#2\def\@currentlabel{#2}\label{#1}}
\makeatother

\begin{document}
\title[The Layered Catalan Monoids: Structure and Determinants]{The Layered Catalan Monoids: Structure and Determinants}
\author{M.H. Shahzamanian}
\address{M.H. Shahzamanian\\ CMUP, Departamento de Matemática, Faculdade de Ciências, Universidade do Porto, Rua do Campo Alegre s/n, 4169--007 Porto (Portugal).}
\email{m.h.shahzamanian@fc.up.pt}

\begin{abstract}
In this paper, we introduce and study a class of monoids, called Layered Catalan Monoids (\( {LC}_n \)), which satisfy the structural conditions for $\ll$-smoothness as defined in~\cite{Sha-Det2}. These monoids are defined by specific identities inspired by Catalan monoids. We establish their canonical forms and compute their determinant, proving that it is non-zero for \(1 \leq n \leq 7\) but vanishes for \(n \geq 8\).
\end{abstract}
\maketitle


\section{Introduction}

In the 1880s, Dedekind introduced the concept of the group determinant for finite groups, which was later studied in depth with Frobenius. 
At the same time, Smith approached the concept from a different perspective, focusing on the determinant of a \(G \times G\) matrix where the entry at position \((g, h)\) is \(x_{gh}\), with \(G\) being a finite group and \(x_k\) as variables for each \(k \in G\) \cite{Smith}. 
This concept has since been extended to finite semigroups, with various researchers contributing to its development \cite{Lindstr, Wilf, Wood}. 
A notable application of the semigroup determinant arises in coding theory. The extension of the MacWilliams theorem from codes over finite fields to chain rings relies on the non-zero determinant property of certain semigroups. Specifically, the theory of linear codes over finite Frobenius rings demonstrates this extension property, where the semigroup determinant plays a key role \cite{Wood-Duality}.  


In \cite{Ste-Fac-det}, Steinberg examined the factorization of the semigroup determinant for commutative semigroups, proving that it is either zero or factors into linear polynomials. Building on this, in \cite{Sha-Det}, the determinant of semigroups in the pseudovariety \(\pv{ECom}\) is studied. This pseudovariety consists of semigroups whose idempotents commute. As shown by Ash, \(\pv{ECom}\) is generated by finite inverse semigroups--a broader class than the semigroups with central idempotents considered in \cite{Ste-Fac-det}.

More recently, in \cite{Sha-Det2}, the study was extended to semigroups beyond \(\pv{ECom}\), focusing on finite semigroups possessing pairs of non-commutative idempotents. The goal was to determine when the determinant is non-zero and to explore its factorization. The focus was on a class of semigroups that, while not in \(\pv{ECom}\), satisfy certain structural properties such as $\ll$-smoothness, $\ll$-transitivity, and being a singleton-rich monoid. These structural conditions were introduced and explored in \cite{Sha-Det2}.  

However, a fundamental question remains open: Are there other natural classes of semigroups or monoids that satisfy these conditions, and what are their determinant properties?

In this paper, we introduce and investigate a class of monoids, which we call Layered Catalan Monoids (\(LC_n\)), satisfying the conditions for $\ll$-smoothness as defined in \cite{Sha-Det2}. These monoids are characterized by specific algebraic identities, including those defining Catalan monoids, with some exceptions. An integer from \(\mathbb{N}\) represents the number of generators of the monoid, excluding the identity element. (Including the identity, the total number of generators is \(n+1\).)

We determine the canonical form of the elements in \(LC_n\), which is essential for verifying the $\ll$-smoothness property. Using the theoretical framework from \cite{Sha-Det2}, we then compute the determinant of \(LC_n\). Our main result shows that the determinant is non-zero for \(1 \leq n \leq 7\) but becomes zero for \(n \geq 8\). This marks an interesting transition point in the determinant behavior of these monoids.  

The paper is structured as follows. In Section~2, we provide preliminaries on semigroup identities and determinant properties. In Section~3, we define Layered Catalan Monoids and study their structural properties, leading to the proof that they are $\ll$-smooth. Finally, in Section~4, we compute the determinant of \(LC_n\) and discuss its implications.  


\section{Preliminaries}



\subsection{Identities}

For standard notation and terminology related to semigroups and monoids, we refer the reader to~\cite[Chapter 5]{Alm}, ~\cite[Chapters 1–3]{Cli-Pre}, and ~\cite[Appendix A]{Rho-Ste}.  
For identities, we refer to~\cite{MR912694}.

Let $X$ be a countably infinite set. 
We call $X$ an alphabet and each element $x\in X$ a letter.
Let $X^+$ be the set of all finite, non-empty words $x_1\cdots x_n$ with $x_1,\ldots,x_n\in X$.
The set $X^+$ forms a semigroup under concatenation which is called the free semigroup over $X$.
The monoid \( X^* = (X^+)^1 \) is called the \emph{free monoid} over \( X \), with the empty word as its identity element.

Let $t = x_1\cdots x_n$ be a word of $X^+$ with $x_1,\ldots, x_n \in X$. 
The set $\{x_1,\ldots, x_n\}$ is called the content of $t$ and is denoted $c(t)$
while the number $n$ is referred to as the length of $t$ and is denoted $\abs{t}$. 
If $x \in c(t)$, we say that a letter $x$ occurs in a word $t$. 
We say that a word $s \in X^+$ occurs in $t$ if $t = t_1st_2$ for some $t_1, t_2 \in X^*$.
Let $u=u_1\cdots u_m$ be a word in $X^*$ with $u_1,\ldots,u_m\in X$.
We say that $u$ is a subword of the word $t$, 
if $t$ can be written $t = u'_0u_1u'_1\cdots u_mu'_m$ for some words $u'_0,u'_1,\ldots,u'_m \in X^*$. 
For a subset $Y$ of the set $X$, let $t_Y$ be the longest subword of $t$ with $c(t_Y)\subseteq Y$.

An identity is an expression $t_1=t_2$ with $t_1,t_2\in X^*$.
Let $M$ be a monoid.
We say that the identity $t_1=t_2$ holds in $M$ or $M$ satisfies the identity $t_1=t_2$ if $\phi(t_1) = \phi(t_2)$ for every homomorphism $\phi\colon X^*\rightarrow M$. 


\subsection{Determinant of a semigroup}
For standard notation and terminology relating to finite dimensional algebras, the reader is referred to \cite{Assem-Ibrahim, Benson}.

A based algebra is a finite-dimensional complex algebra \( A \) equipped with a distinguished basis \( B \). We often refer to the pair \( (A, B) \) to emphasize both the algebra and its basis.  
The multiplication in \( A \) is determined by the structure constants with respect to \( B \), given by the equations  
\[
bb' = \sum\limits_{b''\in B} c_{b'',b,b'} b'',
\]
where \( b, b' \in B \) and \( c_{b'',b,b'} \in \mathbb{C} \).  

Let \( X_B = \{ x_b \mid b \in B \} \) be a set of variables in bijection with \( B \). These structure constants can be represented in a matrix called the Cayley table, which is a \( B \times B \) matrix with entries in the polynomial ring \( \mathbb{C}[X_B] \), defined by  
\[
C(A,B)_{b,b'} = \sum\limits_{b''\in B} c_{b'',b,b'} x_{b''}.
\]  
The determinant of this matrix, denoted by $\theta_{(A,B)}(X_B)$, is either identically zero or a homogeneous polynomial of degree $\abs{B}$.


Let \( S \) be a finite semigroup. The semigroup \( \mathbb{C} \)-algebra \( \mathbb{C}S \) consists of all formal sums  
\[
\sum\limits_{s\in S} \lambda_s s,
\]
where \( \lambda_s \in \mathbb{C} \) and \( s \in S \), with multiplication defined by the formula  
\[
\left( \sum\limits_{s\in S} \lambda_s s \right) \cdot \left( \sum\limits_{t\in S} \mu_t t \right) = \sum\limits_{u=st\in S} \lambda_s \mu_t u.
\]
Note that \( \mathbb{C}S \) is a finite-dimensional \( \mathbb{C} \)-algebra with basis \( S \).  

If we set \( A = \mathbb{C}S \) and \( B = S \), then the Cayley table \( C(S) = C(\mathbb{C}S, S) \) is the \( S \times S \) matrix over \( \mathbb{C}[X_S] \) with  
\[
C(S)_{s,s'} = x_{ss'},
\]
where \( X_S = \{ x_s \mid s \in S \} \) is a set of variables in bijection with \( S \).  
We denote the determinant \( \DDet C(\mathbb{C}S, S) \) by \( \theta_S(X_S) \) and refer to it as the (Dedekind-Frobenius) \emph{semigroup determinant} of \( S \).  
If the semigroup \( S \) is fixed, we often write \( X \) instead of \( X_S \).  

For more details on this topic, we refer the reader to \cite{Frobenius1903theorie}, \cite[Chapter~16]{Okn}, and \cite{Ste-Fac-det}.  


The \emph{contracted semigroup algebra} of a semigroup \( S \) with a zero element \( 0 \) over the complex numbers is defined as  
\[
\mathbb{C}_0S = \mathbb{C}S / \mathbb{C}0;
\]
note that \( \mathbb{C}0 \) is a one-dimensional two-sided ideal.  
This algebra can be thought of as having a basis consisting of the non-zero elements of \( S \) and having multiplication that extends that of \( S \), but with the zero of the semigroup being identified with the zero of the algebra.  

The \emph{contracted semigroup determinant} of \( S \), denoted by \( \widetilde{\theta}_S \), is the determinant of  
\[
\widetilde{C}(S) = C(\mathbb{C}_0S, S\setminus\{0\}),
\]
where  
\[
\widetilde{C}(S)_{s,t} =
\begin{cases}
    x_{st}, & \text{if } st \neq 0, \\
    0, & \text{otherwise}.
\end{cases}
\]
Let \( \widetilde{X} = X_{S \setminus \{0\}} \) if \( S \) is understood.  

According to Proposition 2.7 in~\cite{Ste-Fac-det} (the idea mentioned in~\cite{Wood}), there is a connection between the contracted semigroup determinant and the semigroup determinant of a semigroup \( S \) with a zero element.  
There is a \( \mathbb{C} \)-algebra isomorphism between the \( \mathbb{C} \)-algebra \( \mathbb{C}S \) and the product algebra \( \mathbb{C}_0S \times \mathbb{C}0 \), which sends \( s \in S \) to \( (s, 0) \).  
Put \( y_s = x_s - x_0 \) for \( s \neq 0 \) and let \( Y = \{ y_s \mid s \in S \setminus \{0\} \} \).  
Then  
\[
\theta_S(X) = x_0 \widetilde{\theta}_S(Y).
\]  
Therefore, \( \widetilde{\theta}_S(\widetilde{X}) \) can be obtained from \( \theta_S(X) / x_0 \) by replacing \( x_0 \) with 0.  


\section{Layered Catalan monoid}

The \textit{Catalan monoid}, denoted $C_n$ for each $n \geq 1$, is defined as the monoid generated by the elements $a_1, a_2, \ldots, a_n$	and an identity element $1$, subject to the following relations:
\begin{equation}\label{Cn-relations}
\begin{aligned}
    &a_i^2 = a_i, &&\quad \forall\ 1 \leq i \leq n,\\
    &a_i a_j = a_j a_i, &&\quad \forall\ 1 \leq i, j \leq n, \ \text{if } \abs{i - j} > 1,\\
    &a_i a_{i+1} a_i = a_{i+1} a_i a_{i+1} = a_{i+1} a_i, &&\quad \forall\ 1 \leq i < n.
\end{aligned}
\end{equation}
The monoid \( C_n \) is isomorphic to the monoid of transformations \( f \) on the set \(\{1, 2, \ldots, n\}\) that are both decreasing and order-preserving. Specifically, \( f(i) \leq i \) for all \( i \), and \( f(i) \leq f(j) \) whenever \( i < j \) (see \cite{Sol-MR1406781}). 

Moreover, the cardinality of \( C_n \) is equal to the \( n \)-th Catalan number, given by 
\[
\frac{1}{n+1} \binom{2n}{n},
\]
as established in \cite{Hig-MR1198412}.

We define a class of monoids closely related to the Catalan monoids, called the \textit{Layered Catalan monoid}, denoted by ${LC}_n$ for each \( n \geq 1 \). The monoid \( {LC}_n \) is generated by the elements \( a_1, a_2, \ldots, a_n \) along with the identity element \( 1 \), which represents the empty word. It is defined by the relations in~\eqref{Cn-relations}, except for the relation \( a_1 a_n = a_n a_1 \), along with the relation  
\[
a_n a_1 a_n = a_1 a_n a_1 = a_1 a_n,
\]  
which we denote as (1'),
and the additional relations:  
\begin{equation}\label{LCn-relations}
\begin{aligned}
    a_i a_{(i+1) \bmod n} a_{(i+2) \bmod n} &= a_i a_{(i+2) \bmod n}, \quad &&\forall\ 1 \leq i \leq n.
\end{aligned}
\end{equation}
Note that the relation  
\[
a_n a_1 a_n = a_1 a_n a_1 = a_1 a_n
\]  
is not explicitly included in \eqref{Cn-relations} and must be considered separately.

This additional relation simplifies interactions involving three consecutive generators by collapsing intermediate elements.

\begin{lem}\label{relation3}
In the monoid \( LC_n \), where \( n \geq 1 \),
the following identities hold:
\begin{equation}\label{LCn-relations-res1}
\begin{aligned}
    a_i a_{i+2} &= a_i a_{i+1} a_{i+2} = a_i a_{i+2} a_{i+1}\\
    &= a_{i+1} a_i a_{i+2} = a_{i+1} a_{i+2} a_i\\ &= a_{i+2} a_i a_{i+1} = a_{i+2} a_{i+1} a_i, \quad \forall\ 1 \leq i \leq n.
\end{aligned}
\end{equation}
Note that indices are taken modulo \( n \), meaning that if \( i+1 > n \) or \( i+2 > n \), we interpret them as \( (i+1) \bmod n \) and \( (i+2) \bmod n \), respectively.
\end{lem}

\begin{proof}
By relation \eqref{LCn-relations}, we have \(a_i a_{i+2} = a_i a_{i+1} a_{i+2}\).

We prove other equalities using the relations in (1') and \eqref{LCn-relations}. 

First, consider \( a_i a_{i+2} a_{i+1} \):
\[a_i a_{i+2} a_{i+1} = a_i a_{i+2} a_{i+1} a_{i+2} = a_{i+2} a_i a_{i+1} a_{i+2} = a_{i+2} a_i a_{i+2} = a_i a_{i+2} a_{i+2}=  a_i a_{i+2}\]

Now, consider \( a_{i+1} a_i a_{i+2} \):
\[a_{i+1} a_i a_{i+2} = a_i a_{i+1} a_i a_{i+2} = a_i a_{i+1} a_{i+2} a_i  = a_i a_{i+2} a_i = a_i a_i a_{i+2} = a_i a_{i+2}.\]

As \( a_i a_{i+2}= a_{i+2} a_i\), we have  \(a_{i+1} a_{i+2} a_i = a_{i+1} a_i a_{i+2} = a_i a_{i+2}\) and 
\(a_{i+2} a_i a_{i+1} = a_i a_{i+2} a_{i+1} = a_i a_{i+2}\).

Also, as \(a_{i+2} a_{i+1} a_{i+2} = a_{i+2} a_{i+1}\) and \(a_{i+1} a_{i+2} a_i = a_i a_{i+2}\), 
we have \[a_{i+2} a_{i+1} a_i = a_{i+2} a_{i+1} a_{i+2} a_i = a_{i+2}a_i a_{i+2}=a_i a_{i+2}.\]
\end{proof}

Let \(\Lambda\) be the set of relations given by (1'), \eqref{LCn-relations}, and \eqref{LCn-relations-res1}.  
We construct a canonical form for words over a finite alphabet \( A = \{a_1, \ldots, a_n\} \), where \( n \geq 1 \), with respect to \(\Lambda\). In fact, we manage to show that the equality \(u = v\) for distinct canonical forms does not follow from \(\Lambda\).

Let \( W = w_1 w_2 \cdots w_{\abs{W}} \) be a word over \( A \), where \( w_1, \ldots, w_{\abs{W}} \in A \). Suppose that \(W\) satisfies the following condition:
\begin{enumerate}
\item[\mylabel{itm:star}{($\star$)}] There exists an integer \( 1 \leq r \leq n \) such that neither \( a_{(r+1) \bmod n} \) nor \( a_{(r+2) \bmod n} \) belongs to \( c(W) \). 
\end{enumerate}
In this case, we partition the set \( c(W) \) into disjoint subsets:  
\[
c(W) = C_1 \cup C_2 \cup \cdots \cup C_m
\]
where each subset \( C_j = \{ a_{j_1}, a_{j_2}, \ldots, a_{j_{\abs{C_j}}} \} \) satisfies the following conditions:
\begin{enumerate}
    \item For each \( 1 \leq \alpha \leq \abs{C_1} - 1 \), we have either \( 1_{\alpha+1} = 1_{\alpha} + 1 \bmod n \) or \( 1_{\alpha+1} = 1_{\alpha} + 2 \bmod n \);
    \item For each \( 1 < j \leq m \) and \( 1 \leq \alpha \leq \abs{C_j} - 1 \), we have either \( j_{\alpha+1} = j_{\alpha} + 1 \) or \( j_{\alpha+1} = j_{\alpha} + 2 \);
    \item For every \( 1 \leq j \leq m-1 \), the last index of \(C_j\) satisfies \( j_{\abs{C_j}} < (j+1)_1 - 2 \), where \((j+1)_1\) denotes the first index of the subset \(C_{j+1}\).
\end{enumerate}
If there exists an integer \( 1 \leq \alpha \leq \abs{C_1} - 1 \) such that \( 1_{\alpha+1} < 1_{\alpha} \), we adjust the indices of \(A\) to maintain order by shifting them as follows:
\[
i \mapsto i - 1_{1} + 1\bmod n.
\]
After shifting, we ensure that the largest integer \( r \) in \( c(W) \) satisfies the condition that neither \( a_{(r+1) \bmod n} \) nor \( a_{(r+2) \bmod n} \) belongs to \( c(W) \).

A canonical form for \( W \) is constructed by applying a reduction rule, which operates under the assumption that a total order is fixed for the underlying alphabet \( A \).  After obtaining our canonical form, we shift the indices of the elements of \( A \) back to the original form: 
\[
i \mapsto i + 1_{1} - 1\bmod n.
\]
Afterward, we discuss the construction of a canonical form for words \( W \) that do not satisfy the given condition \ref{itm:star}.

To construct the canonical form, we initialize two variables:  
\begin{itemize}
    \item \( Y \): an initially empty word, which will store the canonical form.
    \item \( X \): initialized as \( W \), the input word, which is reduced step by step.
\end{itemize}

\section*{Reduction Process}
The reduction process involves repeatedly applying the reduction rule until \( X \) becomes the empty word. At the end of this process, \( Y \) will contain the canonical form of \( W \). 

As part of the Reduction Rule, the process Simplify \( Y \) is applied whenever \( Y \) is updated. The rules for simplifying \( Y \) are detailed below.

\section*{Simplify \( Y \)}
\begin{enumerate}
\item[(a)] If \( Y \) matches :
    \[
    Y = a_r a_{r+1} a_{r+2}Y', \quad Y = a_r a_{r+2} a_{r+1}Y',
    \]
    update \( Y \) as:
    \(
    Y := a_r a_{r+2}Y'.
    \)
\item[(b)] If \( Y \) matches:
    \[
    Y = a_r a_{r+2} a_{r+3}Y', \quad Y = a_r a_{r+3} a_{r+2}Y',
    \]
    update \( Y \) as:
    \(
    Y := a_r a_{r+1} a_{r+3}Y'.
    \)
\item[(c)] If \( Y \) matches:
    \[
    Y = a_r a_{r-1}a_{r+1}Y', \quad Y = a_{r-1} a_r a_{r+1}Y',
    \]
    update \( Y \) as:
    \(
    Y := a_{r-1} a_{r+1} Y'.
    \)
\item[(d)] If \( Y \) matches:
    \[
    Y = a_r a_{r-1}a_{r+2}Y', 
    \]
    update \( Y \) as:
    \(
    Y := a_{r-1} a_ra_{r+2} Y'.
    \)    
\item[(e)] If \( Y \) matches:
    \[
    Y = a_{r-1} a_r a_{r+1} a_{r+2}Y', \quad Y = a_{r-1} a_r a_{r+2} a_{r+1}Y',
    \]
    update \( Y \) as:
    \(
    Y := a_{r-1} a_r a_{r+2} Y'.
    \)
\item[(f)] If \( Y \) matches:
    \[
    Y = a_r a_{r-1}a_{r+1} a_{r+2}Y', \quad Y = a_r a_{r-1}a_{r+2} a_{r+1}Y',
    \]
    update \( Y \) as:
    \(
    Y := a_r a_{r-1} a_{r+2} Y'.
    \)
\item[(g)] If \( Y \) matches:
    \[
    Y = a_r a_{r-1}a_{r+2} a_{r+3}Y', \quad Y = a_r a_{r-1}a_{r+3} a_{r+2}Y',
    \]
    \[
    Y = a_{r-1} a_r a_{r+2} a_{r+3}Y', \quad Y = a_{r-1} a_r a_{r+3} a_{r+2}Y',
    \]
    update \( Y \) as:
    \(
    Y := a_{r-1} a_{r+1} a_{r+3} Y'.
    \)
\end{enumerate}

\section*{Reduction Rule \((\text{rr})\)}
The Reduction Rule operates as follows:
\begin{enumerate}

    \item Identify \( r \), the largest integer such that \( a_r \in c(X) \), where \( X \) is the current word being reduced.
    \item Depending on the presence of \( a_{r-1} \) or \( a_{r-2} \) in \( c(X) \), one of three cases applies:
    \begin{itemize}
        \item \textbf{Case 1:} If \( a_{r-1} \notin c(X) \):
        \begin{itemize}
            \item Update \( X := X_{c(X) \setminus \{a_r\}} \) and \( Y := a_r Y \).
            \item Simplify \( Y \).
        \end{itemize}
        
        \item \textbf{Case 2:} If \( a_{r-1} \in c(X) \) and \( a_{r-2} \notin c(X) \):
        \begin{itemize}
            \item Check if there exist \( 1 \leq r_1 < r_2 \leq n \) such that \( w_{r_1} = a_r \) and \( w_{r_2} = a_{r-1} \):
            \begin{itemize}
                \item If yes, update \( X := X_{c(X) \setminus \{a_r, a_{r-1}\}} \), \( Y := a_r a_{r-1} Y \).
                \item Otherwise, update \( X := X_{c(X) \setminus \{a_r, a_{r-1}\}} \), \( Y := a_{r-1} a_r Y \).
            \end{itemize}
            \item Simplify \( Y \).
        \end{itemize}

        \item \textbf{Case 3:} If \( a_{r-1}, a_{r-2} \in c(X) \):
        \begin{itemize}
            \item Update \( X := X_{c(X) \setminus \{a_r, a_{r-1}\}} \) and \( Y := a_r Y \).
            \item Simplify \( Y \).
        \end{itemize}
    \end{itemize}
\end{enumerate}

We prove that the equation \( W = XY \) holds at the end of every step of the reduction process.

\textbf{Initial State}:\\
Before any reduction step is applied:
\begin{itemize}
    \item \( W = X \), since \( Y \) is initialized as the empty word.
\end{itemize}
Clearly, \( W = XY \) holds in this state.

\textbf{Simplify \( Y \)}:\\
We provide a detailed verification for items (d) and (g) below. The remaining cases can be verified in a similar manner using relations~(1') and~\eqref{LCn-relations-res1}.

First, suppose that \( Y \) matches to the following forms
    \[
    Y = a_r a_{r-1}a_{r+2}Y'', 
    \]
    for some word \( Y'' \). Let
    \(
    Y' = a_{r-1} a_ra_{r+2} Y''.
    \)  
By relations~(1') and~\eqref{LCn-relations-res1}, we have
\begin{align*}
Y &= a_r a_{r-1}a_{r+2}Y''=a_r a_{r+2}a_{r-1}Y''=a_ra_{r+1} a_{r+2}a_{r-1}Y''=a_ra_{r+1} a_{r-1}a_{r+2}Y''\\
  &=a_{r+1} a_{r-1}a_{r+2}Y''=a_{r+1} a_{r-1}a_ra_{r+2}Y''= a_{r-1}a_{r+1}a_ra_{r+2}Y''\\
  &= a_{r-1}a_ra_{r+2}Y''=Y'.
\end{align*}

			Now, suppose that \( Y \) matches to one of the following forms:
			\[
			Y = a_r a_{r-1}a_{r+2} a_{r+3}Y'', \quad Y = a_r a_{r-1}a_{r+3} a_{r+2}Y'',
			\]
			\[
			Y = a_{r-1} a_r a_{r+2} a_{r+3}Y'', \quad Y = a_{r-1} a_r a_{r+3} a_{r+2}Y'',
			\]
			for some word \( Y'' \). Let \( Y' = a_{r-1} a_{r+1} a_{r+3} Y'' \).

			We assume that \( Y = a_r a_{r-1}a_{r+2} a_{r+3}Y'' \). Again by relations~(1') and~\eqref{LCn-relations-res1}, we have:
			\begin{align*}
			Y &= a_r a_{r-1}a_{r+2} a_{r+3}Y'' = a_r a_{r-1}a_r a_{r+2} a_{r+3}Y'' \\
			&= a_r a_{r-1}a_r a_{r+1}a_{r+2} a_{r+3}Y'' = a_r a_{r-1}a_r a_{r+1} a_{r+3}Y'' \\
			&= a_r a_{r-1}a_{r+1} a_{r+3}Y'' = a_{r-1}a_{r+1} a_{r+3}Y''=Y'.
			\end{align*}
			Similarly, for the other cases, we conclude that \( Y' = Y \).

\textbf{Reduction Step}:\\
Suppose \( W = XY \) holds after \( k \) reduction steps. For the \( (k+1) \)-th step:
\begin{itemize}
    \item Let \( r \) be the largest integer such that \( a_r \in c(X) \).
    \item Let \( X' \) and \( Y' \) represent the updated variables after this step.
\end{itemize}
We now verify that $W=X'Y'$ in each of the following cases:
    \begin{itemize}
        \item \textbf{Case 1}: If \( a_{r-1} \notin c(X) \):\\
        		Since \( a_{r-1} \notin c(X) \), by relation~(1'), 
    			we know \( a_r^2 = a_r \) and the letter \( a_r \) commutes with all elements of \( c(X) \). 
    			Thus, the reduction preserves \( X = X'a_r \) and, thus, we have \( XY = X'Y' \).    
        \item \textbf{Case 2}: If \( a_{r-1} \in c(X) \) and \( a_{r-2} \notin c(X) \):\\
			The letter \( a_r \) commutes with all elements of \( c(X) \) except \( a_{r-1} \). 
			Since \( a_{r-2} \notin c(X) \), by relation~(1'), 
			the letter \( a_{r-1} \) commutes with all elements of \( c(X) \) except \( a_r \). 
			Therefore, we can write:
			\[
				X = X_{c(X) \setminus \{a_r, a_{r-1}\}} X_{\{a_r, a_{r-1}\}}.
			\]
			Now, using relation~(1'), \( a_r a_{r-1} a_r = a_{r-1} a_r a_{r-1} = a_r a_{r-1} \), we conclude:
			\begin{itemize}
    				\item If there exist \( 1 \leq r_1 < r_2 \leq n \) such that \( w_{r_1} = a_r \) and \( w_{r_2} = a_{r-1} \):
    					\begin{itemize}
       				 \item Then \( X = X_{c(X) \setminus \{a_r, a_{r-1}\}} a_r a_{r-1} \).
    					\end{itemize}
    				\item Otherwise:
    					\begin{itemize}
        				\item \( X = X_{c(X) \setminus \{a_r, a_{r-1}\}} a_{r-1} a_r \).
    					\end{itemize}
				\end{itemize}
			Thus, in both subcases, \( XY = X'Y' \) is preserved.
			
        		Update \( Y := Y' \).

        \item \textbf{Case 3}: If \( a_{r-1}, a_{r-2} \in c(X) \):\\
        First, replace any \( a_{r-2}^2 \) in \( X \) with \( a_{r-2} \) using relation~(1').
		Redefine \( X' \) to reflect the updated word.
        Since \( a_{r-2} \in c(X) \), we can decompose \( X' \) as:
		\[
		X' = X^{(1)} a_{r-2} X^{(2)} \cdots a_{r-2} X^{(m)},
		\]
		where \( X^{(i)} \) are subwords of \( X' \) such that \( a_{r-2} \notin c(X^{(i)}) \) for all \( i \). 
		The subwords \( X^{(1)} \) and \( X^{(m)} \) may potentially be empty.
		
		As \( a_{r-2} \notin c(X^{(i)}) \), we update \( X^{(i)} \) as 
		\[ X'^{(i)} {X^{(i)}}_{c(X^{(i)}) \setminus \{a_r, a_{r-1}\}}, \]
		for every \( 2 \leq i \leq m \), where \( c(X'^{(i)}) \subseteq \{a_r, a_{r-1}\} \). 
		This means we push forward to the left all occurrences of \( a_r \) and \( a_{r-1} \) in \(X^{(i)}\), 
		ensuring that \( a_r \) and \( a_{r-1} \) do not cross each other. 
		This reordering is achieved using relation~\eqref{Cn-relations}. 
		The word \( X'^{(i)} \) may be empty, or consist of \( a_r \), \( a_{r-1} \), \( a_{r-1}a_r \), or \( a_ra_{r-1} \).
		Similarly, we update \( X^{(1)} \) by pushing all occurrences of \( a_r \) and \( a_{r-1} \) to the right within the subword, 
		ensuring the same ordering constraints hold.
		
		For \( 2 \leq i \leq m \), if \( X^{(i)} \) is either \( a_{r-1}a_r \) or \( a_ra_{r-1} \), 
		then by using relation~\eqref{LCn-relations-res1}, we obtain:
		\[
		a_{r-2}X^{(i)} = a_{r-2}a_r.
		\]
		Similarly, if \( X^{(1)} \) is either \( a_{r-1}a_r \) or \( a_ra_{r-1} \), then:
		\[
		X^{(1)}a_{r-2} = a_{r-2}a_r.
		\]


		Considering these updates, 
		we refine \( X' \) such that \( X'^{(i)} \) may be empty or consist of \( a_r \) or \( a_{r-1} \) (for \( 1 \leq i \leq m \)).
		It follows that all letters \( a_{r-1} \) of $X$ connect to the letter \( a_{r-2} \) in $X$.
		
		If none letter of \( X'^{(i)} \) is equal to \( a_{r-1} \), 
		then the letter \( a_r \) commutes with all letters of \( X' \), and the result follows.

		Now, suppose that \( a_{r-1} \in c(X') \). Since \( a_r \in c(X') \), 
		we move the letter \( a_r \) to the closest occurrence of \( a_{r-1} \). 
		As the letter \( a_{r-1} \) is already connected to the letter \( a_{r-2} \), 
		by relation~\eqref{LCn-relations-res1} we can omit the letter \( a_{r-1} \). 

		We repeat this process until \( X' \) no longer contains any occurrence of \( a_{r-1} \). 
		
		Hence, \( X' = X_{c(X)\setminus\{a_r,a_{r-1}\}}a_r \) and, thus, \( XY = X'Y' \) is preserved. 
    \end{itemize}

The reduction process ends when \( X \) becomes the empty word:
\[
X = \emptyset, \quad Y = W.
\]
At this point, \( W = XY \) holds trivially.

By induction on the number of reduction steps, \( W = XY \) is maintained throughout the process, ensuring the correctness of the reduction process.

Hence, for every non-empty word \( W \) over \( A \), we define its canonical form as follows:  
\[ 
W = W_1 W_2 \cdots W_m,
\]  
where for every \( 1 \leq i \leq m \), there exist integers \( 1 \leq j_i \) and \( 0 \leq i_l \) such that the following conditions are satisfied:  
\begin{enumerate}
    \item[(CF1)] \( W_i \) is one of the following forms:  
    \[
    W_i = a_{j_i}a_{j_i+2}\cdots a_{j_i+2i_l}, \quad 
    W_i = a_{j_i}a_{j_i+1}a_{j_i+3}\cdots a_{j_i+2i_l+1}, \quad 
    W_i = a_{j_i+1}a_{j_i}.
    \]
    \item[(CF2)] If \( W_i = a_{j_i}a_{j_i+1} \) or \( W_i = a_{j_i+1}a_{j_i} \), then \( j_i + 4 \leq j_{i+1} \).
    \item[(CF3)] If \( W_i = a_{j_i}a_{j_i+2}\cdots a_{j_i+2i_l} \), then \( j_i+2i_l + 3 \leq j_{i+1} \).
    \item[(CF4)] If \( W_i = a_{j_i}a_{j_i+1}a_{j_i+3}\cdots a_{j_i+2i_l+1} \), then \( j_i+2i_l + 4 \leq j_{i+1} \).
\end{enumerate}
The conditions (CF2) to (CF4) ensure that there is a gap of at least two consecutive letters between the range of letters in \( W_i \) and \( W_{i+1} \).


For each \( W_i \), we define the letters \( \overline{a}_{W_i} \) and \( \underline{a}_{W_i} \) as follows:
\begin{enumerate}
    \item[•] If \( W_i = a_{j_i}a_{j_i+2}\cdots a_{j_i+2i_l} \), then \( \overline{a}_{W_i} = a_{j_i} \) and \( \underline{a}_{W_i} = a_{j_i+2i_l} \).
    \item[•] If \( W_i = a_{j_i}a_{j_i+1}a_{j_i+3}\cdots a_{j_i+2i_l+1} \), then \( \overline{a}_{W_i} = a_{j_i} \) and \( \underline{a}_{W_i} = a_{j_i+2i_l+1} \).
    \item[•] If \( W_i = a_{j_i+1}a_{j_i} \), then \( \overline{a}_{W_i} = a_{j_i} \) and \( \underline{a}_{W_i} = a_{j_i+1} \).
\end{enumerate}

To prove the uniqueness of the canonical form of a non-empty word \( W \) over \( A \), we assume that \( W \) has two distinct canonical forms that satisfy the conditions described earlier:
\[
W = W_1 \cdots W_{m_1} = V_1 \cdots V_{m_2}.
\]
Here, \(\overline{W}:= W_1 \cdots W_{m_1} \) and \(\overline{V}:= V_1 \cdots V_{m_2} \) represent two different decompositions of \( W \) into subwords.

For each subword \( W_i \), there exist integers \( 1 \leq j_i \) and \( 0 \leq i_l \) such that \( W_i \) is one of the following forms:  
\[
W_i = a_{j_i}a_{j_i+2}\cdots a_{j_i+2i_l}, \quad 
W_i = a_{j_i}a_{j_i+1}a_{j_i+3}\cdots a_{j_i+2i_l+1}, \quad 
W_i = a_{j_i+1}a_{j_i}.
\]
Similarly, for each subword \( V_i \), there exist integers \( 1 \leq k_i \) and \( 0 \leq i_{l'} \) such that \( V_i \) is one of the following forms:  
\[
V_i = a_{k_i}a_{k_i+2}\cdots a_{k_i+2i_{l'}}, \quad 
V_i = a_{k_i}a_{k_i+1}a_{k_i+3}\cdots a_{k_i+2i_{l'}+1}, \quad 
V_i = a_{k_i+1}a_{k_i}.
\]
Both decompositions must also satisfy the other stated conditions and properties of the canonical form.

Let \( r \) and \( r' \) be the smallest integers such that \( W_r \) does not match \( V_{r'} \), \(m_1 - r = m_2 - r'\) and \( W_{r+p} \) matches \( V_{r'+p} \), for \( 1 \leq p \leq m_1 - r \). 

By symmetry, we may assume that the index of the letter \( \overline{a}_{W_r} \) in the set \( A \) is greater than that of \( \overline{a}_{V_{r'}} \).
Since there is a gap of at least two consecutive letters between the letters in \( V_{r'} \) and \( V_{r'+1} \),
it is not possible to include \( \overline{a}_{W_r} \) in any word equal to \( V_{r'} \) or any other \( V_i \) for \( i \leq r' \).
Similarly, if the index of \( \underline{a}_{W_r} \) does not match the index of \( \underline{a}_{V_{r'}} \), we reach a contradiction. 
Thus, we must have \( \overline{a}_{W_r} = \overline{a}_{V_{r'}} \) and \( \underline{a}_{W_r} = \underline{a}_{V_{r'}} \).
This implies that \( c(W_r) = c(V_{r'}) \). 

Since \( W_r \) does not match \( V_{r'} \), we must have either \( W_r = a_{j_r}a_{j_r+1}  \) and \( V_{r'} = a_{j_r+1}a_{j_r}  \), or vice versa.
By symmetry, we assume the former case.  

Thus, we have shown that the elements of the Layered Catalan monoid \( LC_n \) can be uniquely represented in their canonical form. 
As stated earlier, we now shift the indices of the elements of \( A \) back to their original form using the map that we provided.

The canonical form of any word, 
can be described as follows:


\begin{prop}\label{WSegments}
Let \( W \) be a non-empty word over \( A \) that satisfies Condition~\ref{itm:star}.  
We represent the indices of letters in the canonical form of \( W \) as \( m \) disjoint segments around a circular arrangement of \( n \) positions (indexed from \( 1 \) to \( n \)), such that:  
\begin{itemize}
    \item Each segment \( S_k = [i_k, j_k] \) starts at \( i_k \) and ends at \( j_k \) in a clockwise manner around \( 1 \) to \( n \) with length less than the cycle.
    \item The segments are ordered sequentially in a clockwise manner:
    \[
    i_1 \leq j_1 < i_2 \leq j_2 < \cdots < i_m \leq j_m
    \]
    ensuring a partial cyclic order, which allows segments to wrap around the clock (i.e., in a segment can go from the end of the set back to the beginning, e.g., a segment can be as from \( [\ldots, n, 1, \ldots] \)).
    \item The union of all segments forms a partial cyclic order, where segments follow one another, but the cycle is not completed.
\end{itemize}
Furthermore, the following conditions hold:
\begin{enumerate}
    \item \( a_{i_k}, a_{j_k} \in c(W) \), for all \( 1 \leq k \leq m \).
    \item \( a_{i_{k-2}}, a_{i_{k-1}}, a_{j_{k+1}}, a_{j_{k+2}} \not\in c(W) \), for all \( 1 \leq k \leq m \).
    \item No integer \( l \) exists in any segment such that \( a_l, a_{l+1} \not\in c(W) \).
    \item For \( 1 \leq k \leq m \), if \( l \) is between two consecutive segments, then \( a_l \not\in c(W) \).
\end{enumerate}
Depending on the distance \( j_k - i_k \), in a clockwise direction, we assign the word \( W_k \) as follows:
\begin{itemize}
    \item If \( i_k = j_k \), then \( W_k = a_{i_k} \).
    \item If \( j_k - i_k = 1 \), and there exist \( w_l \) and \( w_{l'} \) (letters of \( W \)) such that \( l' < l \), \( w_l = a_{i_k} \), and \( w_{l'} = a_{j_k} (= a_{i_k+1}) \), then \( W_k = a_{i_k+1} a_{i_k} \). Otherwise, \( W_k = a_{i_k} a_{i_k+1} \).
    \item If \( j_k - i_k \geq 2 \), and \( j_k - i_k \) is even, then \( W_k = a_{i_k} a_{i_k+2} \cdots a_{j_k} \). If \( j_k - i_k \) is odd, then \( W_k = a_{i_k} a_{i_k+1} a_{i_k+3} \cdots a_{j_k} \).
\end{itemize}
The canonical form of \( W \) is equal to \( W_1 \cdots W_m \).
%
\end{prop}


Thus, the structure of \( W \) is formed by a sequence of non-overlapping and well-separated segments on the clock face that follow the partial cyclic order, with potential wrapping of segments but not completing a full cycle.

Note that, as \( \overline{a}_{W_{r-1}} \leq j_{r}+3 \) and \( W \) is in canonical form, it is not possible to include \( a_{j_r+1} \) before \( a_{j_r} \).  

In the canonical form of \( W = W_1 \cdots W_m \), we refer to each \( W_k \) as a segment of \( W \). The indices \( i^{(W)}_k \) and \( j^{(W)}_k \) represent the starting and ending positions of the segment \( W_k \), respectively.  
If the context is clear, we refer to the word \( W \) instead of explicitly mentioning \( i^{(W)}_k \) and \( j^{(W)}_k \), and we write \( i_k \) and \( j_k \).  
When \( W_k = a_{i_k} a_{i_k+1} \), we say that \( W_k \) is a blocker.
Additionally, it is clear that if \( k \neq k' \), the segments \( W_k \) and \( W_{k'} \) commute.

Now, proceed under the assumption that \( W \) does not satisfy Condition~\ref{itm:star}.

We prove that if \( n > 3 \), then any word that does not satisfy condition \ref{itm:star} is equal to the zero element of \( S \), which serves as the absorbing element in \( S \).
However, the cases \( n = 2 \) and \( n = 3 \) behave differently.

If \( n = 2 \), then by (1'), using the relations \( a_1^2 = a_1 \), \( a_2^2 = a_2 \), and  
\[
a_1 a_2 a_1 = a_2 a_1 a_2 = a_1 a_2 = a_2 a_1,
\]
we conclude that if \( c(W) = \{a_1, a_2\} \), then \( W = a_1 a_2 \).  
However, the words \( a_1 \) and \( a_2 \) do not satisfy condition \ref{itm:star}.

If \( n = 3 \), then by Lemma~\ref{relation3}, we apply the following relations:  
\[
a_3 a_2 a_1 = a_3 a_1 a_2 = a_2 a_3 a_1 = a_1 a_2 a_3 = a_1 a_3 a_2 = a_2 a_1 a_3 = a_1 a_3 = a_2a_1 = a_3a_2.
\]
Using these, we can eliminate all occurrences of the letter \( a_2 \) in \( W \) whenever \( c(W) = \{a_1, a_2, a_3\} \).
Thus, in this case, we have \( W = a_1 a_3 \).
However, the words \( a_1 a_2 \), \( a_2 a_3 \), and \( a_3 a_1 \) are pairwise distinct and not equal to \( a_1 a_3 \). Moreover, none of them satisfy condition \ref{itm:star}.
Moreover \( W = a_1 a_3 \) is the zero of $S$.

Now, let us consider the case where \( n > 3 \). 
We prove that if \( n \) is even, then \( W = a_1 a_3 \cdots a_{n-1} \); otherwise, if \( n \) is odd, we have \( W = a_1 a_2 a_4 \cdots a_{n-1} \). First, we consider the case \( n = 4 \), and afterward, we prove the result for all \( n \geq 5 \).

Suppose that \( n = 4 \).  

We first consider the case where \( c(W) \) does not contain a particular letter, say \( a_4 \).  
Since \( W \) does not satisfy condition~\ref{itm:star}, it follows that \( a_1, a_3 \in c(W) \).  
By Lemma~\ref{relation3}, we can eliminate all occurrences of \( a_2 \) in \( W \).  
Hence, we obtain \( W = a_1 a_3 \). 

On the other hand, if \( c(W) \) contains all letters of \( A \),  
then there exist words \( W_1, \ldots, W_m \) such that \( W \) can be expressed as  
\[
W = W_1 a_4 W_2 \cdots a_4 W_m,
\]
where \( c(W_i) \subseteq \{a_1, a_2, a_3\} \) for each \( i \).  

If \( W_i \) is non-empty, we face the following possibilities: 

If \( c(W_i) = \{a_1, a_2, a_3\} \), then by Lemma~\ref{relation3}, we have \( W_i = {W_i}_{\{a_1,a_3\}} \).  Hence,
we can eliminate all occurrences of \( a_2 \) from \( W_i \).

If \( c(W_i) = \{a_1, a_2\} \), then using the relations 
\( a_1a_2a_4=a_1a_2a_3a_4 \), \[a_4a_1a_2 = a_4a_2= a_4a_2a_3=a_4a_1a_2a_3,\]
\(a_2a_1=a_2a_1a_2=a_1a_2a_1\)
 and noting that \( W_i \) is flanked by \( a_4 \) in \( W \), we can insert \( a_3 \) in \( W_i \).  
By the argument above, we can then eliminate all occurrences of \( a_2 \) from \( W_i \). The same reasoning applies when \( c(W_i) = \{a_3, a_2\} \).  

Similarly, if \( W_i = a_2 \), then using the relations \[ a_4a_2=a_2a_4 = a_1a_2a_4 = a_1a_2a_3a_4 ,\] we can substitute \( a_1a_2a_3 \) for \( W_i \) in \( W \) when \( W_i \) is flanked by \( a_4 \).  

Hence, we can eliminate all occurrences of \( a_2 \) in \( W \), and by the argument above, we obtain \( W = a_1 a_3 \).  

Throughout this proof, we assume that the indices of the elements of \(A\) are taken modulo $n$. That is, by \(a_{i-1}\), we mean \(a_{(i-1) \bmod n}\).

Now, we consider the case where \( n \geq 5 \). In this case, for every integer \( i \), the integers  
\( i-3, i-2, i-1, i, i+1 \) are pairwise distinct modulo \( n \).  
Additionally, the distance between \( i+1 \) and \( i-1 \) in one direction modulo \( n \) is 1,  
while in the other direction, it is at least 2.

First, suppose that there exists an integer \( i \) such that \( a_i \not\in c(W) \).  
Since \( W \) does not satisfy condition~\ref{itm:star}, it follows that \( a_{i-1}, a_{i+1} \in c(W) \).  

If \( W = W_1 a_{i-1} a_{i+1} W_2 \) or \( W = W_1 a_{i+1} a_{i-1} W_2 \) for some words \( W_1 \) and \( W_2 \),  
then, by Lemma~\ref{relation3}, we can insert \( a_i \) into the word \( W \).  

Otherwise, we have  
\[
W = W_1 a_{i-1} W_2 a_{i+1} W_3  
\quad \text{or} \quad  
W = W_1 a_{i+1} W_2 a_{i-1} W_3,
\]
where \( W_2 \) is a non-empty word and \( a_{i-1}, a_i, a_{i+1} \not\in c(W_2) \).  

First, assume the former case.  

If \( a_{i-2} \not\in c(W_2) \), then we have \( a_{i-1} W_2 = W_2 a_{i-1} \),  
which allows us to insert \( a_i \) into \( W \),  
since the relation  
\[
a_{i-1}  a_{i+1} = a_{i-1} a_i a_{i+1}
\]
holds.  

Otherwise, we have \( W_2 = W_4 a_{i-2} W_5 \) for some words \( W_4 \) and \( W_5 \),  
where \( a_{i-2} \not\in W_4 \).  
Hence, we obtain \( a_{i-1} W_4 = W_4a_{i-1} \).  

Now, using the relation  
\[
a_{i-1} a_{i-2} = a_{i-1} a_{i-2} a_{i-1},
\]
we can rewrite \( W \) as  
\[
W = W_1 a_{i-1} W_4 a_{i-2} a_{i-1} W_5 a_{i+1} W_3.
\]
By continuing this process, we can progressively adjust \( W_5 \)  
until we insert the letter \( a_{i-1} \) immediately to the left of \( a_{i+1} W_3 \).  
At this stage, we can insert the letter \( a_i \) into \( W \).  

Now, assume that  
\[
W = W_1 a_{i+1} W_2 a_{i-1} W_3.
\]
If \( a_{i-2} \not\in c(W_2) \), then we have  
\[
a_{i+1} W_2 a_{i-1} = a_{i+1} a_{i-1} W_2,
\]
which allows us to insert \( a_i \) into \( W_2 \).  

Otherwise, if \( a_{i-2} \in c(W_2) \), then we consider \( a_{i-3} \).  
If \( a_{i-3} \not\in c(W_2) \), then we obtain  
\[
a_{i+1} W_2 a_{i-1} = a_{i+1} a_{i-2} a_{i-1} (W_2)_{c(W_2)\setminus \{a_{i-2}\}}
= a_{i-2} a_{i+1} a_{i-1} (W_2)_{c(W_2)\setminus \{a_{i-2}\}}.
\]
Thus, we can insert \( a_i \) into \( W \), completing the proof.

Now, suppose that \( a_{i-3} \in c(W_2) \).  
Thus, we can write \( W_2 \) as  
\(
 W_4 a_{i-3} W_5,
\)
where \( a_{i-3} \not\in c(W_5) \).  
Then, we have  
\[
W_2 a_{i-1} = W_4 a_{i-3} W_5 a_{i-1}.
\]  
If \( a_{i-2} \not\in c(W_5) \), we have
\[
W_2 a_{i-1} = W_4 a_{i-1} a_{i-3} W_5.
\]
Otherwise, then we obtain  
\[
W_2 a_{i-1} = W_4 a_{i-3} a_{i-2} a_{i-1} {W'}_5
= W_4 a_{i-3} a_{i-1} {W'}_5 = W_4 a_{i-1} a_{i-3} {W'}_5,
\]
for some word \( {W'}_5 \).  
By iterating this process, we continue shifting \( a_{i-1} \) until it is positioned immediately to the right of \( a_{i+1} \).  
At this stage, we can insert the letter \( a_i \) into \( W \), completing the argument.

Since \( a_i \) was chosen arbitrarily,  
we may assume that \( c(W) = A \).

Then, we have  
\[
W = W_1 a_n W_2 a_n \cdots W_{m-1} a_n W_m,
\]
for some words \( W_1, \ldots, W_m \) where \( a_n \not\in W_i \) for each \( i \).  

Suppose that \( a_1 \in W_i \) for some \( 1 < i < m \).  
Then, we can write \( W_i = V_1 a_1 V_2 \), for some words \( V_1 \) and \( V_2 \) such that \( a_1 \not\in c(V_2) \).  

If \( a_2 \not\in c(V_2) \), then we have \( a_1 V_2 = V_2 a_1 \), and thus,  
\( W_i = {W'}_i a_1 \), for some word \( {W'}_i \).  

Otherwise, if \( a_2 \in c(V_2) \), then we write \( V_2 = V_3 a_2 V_4 \),  
for some words \( V_3 \) and \( V_4 \) where \( a_2 \not\in c(V_3) \).  
Thus, we obtain \( a_1 V_2 = V_3 a_1 a_2 V_4 \).  

If \( a_3 \not\in c(V_4) \), then we get  
\[
W_i a_n = V_1V_3(V_4)_{c(V_4)\setminus\{a_2\}} a_1 a_2 a_n = V_1V_3(V_4)_{c(V_4)\setminus\{a_2\}} a_1 a_n a_2.
\]  

Otherwise, if \( a_3 \in c(V_4) \), then we write \( V_4 = V_5 a_3 V_6 \),  
for some words \( V_5 \) and \( V_6 \) where \( a_3 \not\in c(V_5) \).  
Thus, we obtain  
\[
a_1 a_2 V_4 = a_1 a_2 V_5 a_3 V_6 = V'_5 a_1 a_2 a_3 V_6 = V'_5 a_1 a_3 V_6 = V'_5 a_3 a_1 V_6.
\]  
for some word $V'_5$.

By iterating this process, we continue shifting \( a_1 \) until it is positioned immediately to the left of \( a_n \) in \( W_i a_n \).  
Similarly, we obtain the same result for \( W_1 a_n \).  

Now, suppose again that \( a_1 \in W_i \) for some \( 1 < i < m \).  
Then, we can write \( W_i = V_1 a_1 V_2 \), for some words \( V_1 \) and \( V_2 \) such that \( a_1 \not\in c(V_1) \).  

If \( a_2 \not\in c(V_1) \), then \( V_1 a_1 = a_1 V_1 \), and thus,  
\[
W_i = a_1 {W'}_i,
\]
for some word \( {W'}_i \).  

Otherwise, if \( a_2 \in c(V_1) \), then we write \( V_1 = V_3 a_2 V_4 \),  
where \( V_3 \) and \( V_4 \) are words with \( a_2 \not\in c(V_4) \).  
Then, we obtain  
\[
V_1 a_1 = V_3 a_2 V_4 a_1 = V_3 a_2 a_1 V_4 = V_3 a_1 a_2 a_1 V_4.
\]  

Continuing this process, we shift \( a_1 \) step by step until it appears immediately to the right of \( a_n \) in \( a_n W_i \).  
Similarly, we obtain the same result for \( a_n W_m \).  

Applying the same argument, we achieve an analogous result for \( a_{n-1} \).  

If there exists an integer \( i \) such that \( a_1, a_{n-1} \in c(W_i) \),  
then we can write  
\[
a_n W_i = W''_i a_n a_{n-1} a_1 W'_i
\]
or equivalently,
\[
W_i a_n = W'_i a_n a_1 a_{n-1} W''_i,
\]
for some words \( W'_i \) and \( W''_i \) where \( a_n \not\in c(W'_i) \) and \( a_n \not\in c(W''_i) \).

Thus, we can eliminate \( a_n \). By repeating this process, we successively remove all occurrences of \( a_n \) from \( W \).  

Now, suppose that such an integer \( i \) does not exist.  
Then, there exist distinct integers \( i_1 \) and \( i_2 \) such that \( a_1 \in c(W_{i_1}) \) and \( a_{n-1} \in c(W_{i_2}) \),  
while \( a_1, a_{n-1} \not\in c(W_i) \) for every \( i_1 < i < i_2 \).  
Thus, we have  
\begin{align*}
W_{i_1} a_n W_{i_1+1} a_n W_{i_1+2} \cdots a_n W_{i_2}  
&= W_{i_1} W_{i_1+1} W_{i_1+2} \cdots a_n W_{i_2} \\  
&= W' a_1 a_n W_{i_2} \\  
&= W' a_1 a_n a_{n-1} W'_{i_2},
\end{align*}
for some words \( W' \) and \( W'_{i_2} \).  A similar argument applies if \( a_1 \in c(W_{i_2}) \) and \( a_{n-1} \in c(W_{i_1}) \).  

Now, by the previous argument, we can eliminate all occurrences of \( a_n \).  

If \( n \) is even, we may continue eliminating the elements \( a_{n-2}, \dots, a_2 \).  
Thus, by relation~\ref{Cn-relations}, we obtain  
\[
W = a_1 a_3 \cdots a_{n-1}.
\]  

Otherwise, if \( n \) is odd, we may eliminate all elements \( a_{n-2}, \dots, a_3 \).  
Then, by relation~\ref{Cn-relations}, we have  
\[
W = W' a_4 \cdots a_{n-1},
\]  
for some word \( W' \) where \( c(W') = \{a_1, a_2\} \).  
This implies  
\[
W = a_1 a_2 a_4 \cdots a_{n-1}  
\quad \text{or} \quad  
W = a_2 a_1 a_4 \cdots a_{n-1}.
\]  
If \( W = a_2 a_1 a_4 \cdots a_{n-1} \), then, by relation~(1'), we have  
\begin{align*}
W 
&= a_2 a_1 a_4 \cdots a_{n-1}  
 = a_2 a_1 a_2 a_4 \cdots a_{n-1} 
 = a_2 a_1 a_2 a_3 a_4 \cdots a_{n-1}  \\
&= a_2 a_1 a_3 a_4 \cdots a_{n-1}  
 = a_1 a_2 a_3 a_4 \cdots a_{n-1}  
 = a_1 a_2 a_4 \cdots a_{n-1}.
\end{align*}

Thus, we conclude the following lemma: 

\begin{lem}\label{LCnZero}
For \( n > 3 \), the zero of \( LC_n \) is \( a_1 a_3 \cdots a_{n-1}\) if \( n \) is even,  
and \( a_1 a_2 a_4 \cdots a_{n-1} \) if \( n \) is odd.  
For \( n = 2 \), the zero is \( a_1 a_2 \), and for \( n = 3 \), it is \( a_1 a_3 \).  
\end{lem}

The monoid \( LC_2 \) consists of four elements: \(\{a_1a_2, a_1, a_2, 1\}\).  
The monoid \( LC_3 \) has eight elements:  
\[
\{a_1a_3, a_1, a_2, a_3, a_1a_2, a_2a_3, a_3a_1, 1\}.
\]
Apart from the result in Lemma~\ref{LCnZero}, for \( LC_3 \), the concept of the canonical form for non-zero and non-empty elements remains the same as for \( n \geq 4 \).  
Thus, from now on, when discussing the monoid \( LC_n \), we assume \( n > 2 \), with the only distinction being that the notion of zero differs in the case of \( n = 3 \).

Let \( W = W_1 W_2 \cdots W_m \) be a non-empty, non-zero word over \( A \) in its canonical form.
	
For a segment \( W_k \) with starting and ending indices \( i_k \) and \( j_k \), the ordering of the indices, for example, for \( i_k \leq k' \leq j_k \), is interpreted with respect to the clockwise direction within each segment (i.e., the segments are arranged in a circular, clockwise order).
	
\begin{lem}\label{idemNonBlocker}
Let \( W_k \) be a non blocker segment of \( W \), with starting and ending indices \( i_k \) and \( j_k \), respectively.  
If \( i_k \leq k' \leq j_k \), then \( a_{k'} \) commutes with \( W_k \), i.e., \( a_{k'} W_k = W_k a_{k'} = W_k \).
\end{lem}

\begin{proof}
We consider two cases: \( j_k - i_k > 1 \) and \( j_k - i_k = 1 \).

Case \( j_k - i_k > 1 \):\\ 
In this case, the segment \( W_k \) match to the word \( a_{i_k}a_{i_k+1}a_{i_k+2}\cdots a_{j_k} \).  
We use Lemma~\ref{LCn-relations-res1} multiple times in the following arguments.
\begin{itemize}
    \item[\( k' = i_k \):]  
    Clearly, \( a_{k'} W_k = W_k \).  

    \item[\( k' = i_k + 1 \):]  
    We have:  
    \[
    a_{i_k+1}a_{i_k}a_{i_k+1}a_{i_k+2} = a_{i_k+1}a_{i_k}a_{i_k+2} = a_{i_k}a_{i_k+2} = a_{i_k}a_{i_k+1}a_{i_k+2}.
    \]  
    Hence, \( a_{k'} W_k = W_k \).  

    \item[\( k' > i_k + 1 \):]  
    In this case,  
    \begin{align*} 
    a_{k'} W_k &= a_{k'}a_{i_k}a_{i_k+1}\cdots a_{k'-2}a_{k'-1}a_{k'}\cdots a_{j_k}
    = a_{i_k}a_{i_k+1}\cdots a_{k'-2}a_{k'}a_{k'-1}a_{k'}\cdots a_{j_k}\\
    &= a_{i_k}a_{i_k+1}\cdots a_{k'-2}a_{k'}a_{k'}\cdots a_{j_k}
    = a_{i_k}a_{i_k+1}\cdots a_{k'-2}a_{k'}\cdots a_{j_k}\\
    &= a_{i_k}a_{i_k+1}\cdots a_{k'-2}a_{k'-1}a_{k'}\cdots a_{j_k} = W_k.
    \end{align*}  
\end{itemize}

Case \( j_k - i_k = 1 \):\\ 
In this case, the segment \( W_k \) is equal to the word \( a_{i_k+1}a_{i_k} \).  
\begin{itemize}
    \item[\( k' = i_k \):]  
    We have, \( a_{k'} a_{i_k+1}a_{i_k} = a_{i_k} a_{i_k+1}a_{i_k}=a_{i_k+1}a_{i_k}=W_k \).  

    \item[\( k' = i_k + 1 \):]  
    Clearly, \( a_{k'} W_k = W_k \).
\end{itemize}

Similarly, we have \( W_k a_{k'} = W_k \).
\end{proof}

\begin{lem}\label{idem}
The word \( W \) is idempotent if and only if every segment \( W_k \) is not blocker.
\end{lem}

\begin{proof}
We have \( WW = W_1W_1 W_2W_2 \cdots W_mW_m \). 

If \( W_k \) is not blocker, then by Lemma~\ref{idemNonBlocker}, \( W_kW_k = W_k \). 
Consequently, \( W \) is idempotent. It follows that \(W\) is idempotent

On the other hand, if \( W_k = a_{i_k} a_{i_k+1} \), then \( W_kW_k = a_{i_k+1}a_{i_k} \), which ensures that \( W \) is not idempotent. Then \( W\) is not idempotent.
\end{proof}
 
We define two subsets, \( I_W \) and \( N_W \), of \( \{1, \ldots, m\} \) as follows:  
A segment index \( k \) belongs to \( I_W \) if \( W_k \) is not a blocker.  
On the other hand, \( k \in N_W \) if \( W_k = a_{i_k} a_{i_k+1} \).  

Furthermore, if \( I_W \) is non-empty, define \( W^{(I)} = \prod_{k \in I_W} W_k \). Otherwise, \( W^{(I)} \) is defined as the empty word.  

If \( N_W \) is non-empty, define  
\[
N^{l}_W = \prod_{k \in N_W} a_{i_k} \quad \text{and} \quad N^{r}_W = \prod_{k \in N_W} a_{i_{k}+1}.
\]  
Otherwise, both \( N^{l}_W \) and \( N^{r}_W \) are defined as the empty word.

Let \( W_1 W_2 \cdots W_{m_w} \) and \( V_1 V_2 \cdots V_{m_v} \) denote the canonical forms of non-empty and non-zero \( W \) and \( V \) over \( A \), respectively.  

For integers \( 1 \leq k \leq m_v \) and \( 1 \leq k' \leq m_w \), we say that the segment \( V_k \) is covered by \( W_{k'} \) if \( i^{(W)}_{k'} \leq i^{(V)}_k \leq j^{(V)}_k \leq j^{(W)}_{k'} \).  

\begin{lem}\label{PlusAndStarPrin}
Let \( e = W^{(e)}_1 W^{(e)}_2 \cdots W^{(e)}_{m_e} \) be a non-empty idempotent word in canonical form over \( A \), where \( e \in \varphi^{+}(W) \). 
For every \( 1 \leq k \leq m_e \),
there exists \( 1 \leq k' \leq m_w \) such that the segment \( W^{(e)}_k \) is covered by \( W^{(W)}_{k'} \), and if \( W^{(W)}_{k'} \) is a blocker, then \( W^{(e)}_k = a_{i^{(W)}_{k'}} \).

Similarly, for \( e \in \varphi^{\ast}(W) \), the same holds, but if \( W^{(W)}_{k'} \) is a blocker, then \( W^{(e)}_k = a_{i^{(W)}_{k'}+1} \).
\end{lem}

\begin{proof}
Suppose for \( 1 \leq k \leq m_e \), no such \( 1 \leq k' \leq m_w \) exists satisfying the conditions mentioned in the lemma; then there would exist a segment in \( eW \) that does not match any segment of \( W \), contradicting the assumption that \( e \in \varphi^{+}(W) \).  

Now, assume for \( 1 \leq k \leq m_e \), such \( 1 \leq k' \leq m_w \) exists that satisfies the conditions of lemma and \( W^{(W)}_{k'} \) is a blocker. In this case, \( W^{(W)}_{k'} = a_{i^{(W)}_{k'}}a_{i^{(W)}_{k'}+1} \). If \( W^{(e)}_k \neq a_{i^{(W)}_{k'}} \), since \( W \) is idempotent, it follows that \( W^{(e)}_k \) must be either \( a_{i^{(W)}_{k'}+1}a_{i^{(W)}_{k'}} \) or \( a_{i^{(W)}_{k'}+1} \). This introduces the segment \( a_{i^{(W)}_{k'}+1}a_{i^{(W)}_{k'}} \) in \( eW \), leading to a contradiction.  

Similarly, the result holds when \( e \in \varphi^{\ast}(W) \).  
\end{proof}

\begin{lem}\label{PlusAndStar}
We have \( W^{+} = W^{(I)} N^{l}_W \) and \( W^{\ast} = W^{(I)} N^{r}_W \).
Moreover, we have \( W = W^{+}W^{\ast} \).
\end{lem}

\begin{proof}
In the case where \( N_W \) is empty, by Lemma~\ref{idem}, it follows that \( W \) is idempotent, meaning that \( W^{+} = W^{\ast} = W \). 
Also, we have \( W = W^{(I)} \), and both \( N^{l}_W \) and \( N^{r}_W \) are the empty word. This completes the proof of the lemma for this case.

Now, suppose that \( N_W \) is non-empty. Again, by Lemma~\ref{idem} implies that \( W^{(I)} N^{l}_W \) and \( W^{(I)} N^{r}_W \) are idempotent.

It is clear that $W^{(I)} N^{l}_WW=W$, and thus $W^{(I)} N^{l}_W \in \varphi^{+}(W)$.
Let \( e = W^{(e)}_1 W^{(e)}_2 \cdots W^{(e)}_{m_e} \) be a non-empty idempotent word in canonical form, where \( e \in \varphi^{+}(W) \).  
By Lemma~\ref{PlusAndStarPrin}, we have \[ W^{(e)}_k W^{(I)} N^{l}_W = W^{(I)} N^{l}_WW^{(e)}_k = W^{(I)} N^{l}_W, \] for every \( 1 \leq k \leq m_e \), and thus \( W^{(I)} N^{l}_W \leq e \).  
Therefore, \( W^{+} = W^{(I)} N^{l}_W \).

Similarly, it follows that \( W^{\ast} = W^{(I)} N^{r}_W \).

Finally, it is straightforward to verify that \( W = W^{+} W^{\ast} \).
\end{proof}

It is clear that \( W^{(I)} N^{l}_W = N^{l}_W W^{(I)} \) and \( W^{(I)} N^{r}_W = N^{r}_W W^{(I)} \).

Lemma~\ref{PlusAndStar} implies the following Corollary.

\begin{cor}\label{singletonrich}
The monoid \( LC_n \) is singleton-rich.
\end{cor}

The following lemma examines the conditions under which \( W \ll V \).

\begin{lem}\label{llPrim}
We have \( W \ll V \) if and only if, for every $1\leq k\leq m_v$,
the following conditions hold:
\begin{itemize}
\item if \( k \in I_V \) and \( i^{(V)}_k < j^{(V)}_k \), then there exists an integer \(k'\in I_W\) such that \( V_k \) is covered by \( W_{k'} \).
\item if \( k \in I_V \) and \( i^{(V)}_k = j^{(V)}_k \), then either there exists an integer \(k'\in I_W\) such that \( V_k \) is covered by \( W_{k'} \), or there exists \( k' \in N_W \) such that \( i^{(V)}_k = i^{(W)}_{k'} \) or \( i^{(V)}_k = i^{(W)}_{k'} + 1 \).
\item if \( k \in N_V \), then either there exists an integer \(k'\in I_W\) such that \( V_k \) is covered by \( W_{k'} \), or there exists \( k' \in N_W \) such that \( i^{(V)}_k = i^{(W)}_{k'} \).
\end{itemize}
\end{lem}

\begin{proof}
By Lemma~\ref{PlusAndStar}, we have \( W = W^{+}W^{\ast} = W^{(I)} N^{l}_W N^{r}_W W^{(I)} \).  
Since \( W^{(I)} \) commutes with both \( N^{l}_W \) and \( N^{r}_W \),  
if \( V_k \), for some \( 1 \leq k \leq m_v \), satisfies the conditions specified in the lemma, it follows that  
\[
W = W^{(I)} N^{l}_W N^{r}_W W^{(I)} = W^{(I)} N^{l}_W V_k N^{r}_W W^{(I)}.
\]  
Therefore, we conclude that \( W \ll V \). 

Now, suppose that \( W \ll V \).  

Let \( 1 \leq k \leq m_v \).  
Assume \( j^{(V)}_k - i^{(V)}_k > 2 \).  
If there does not exist an integer \( 1 \leq k' \leq m_W \) such that  
\( i^{(W)}_{k'} \leq i^{(V)}_k \leq j^{(V)}_k \leq j^{(W)}_{k'} \),  
then \( W^{+} V W^{\ast} \neq W \), leading to a contradiction.  
Thus, it must be that such a \( k' \) exists, and consequently, \( V_k \) is covered by \( W_{k'} \).

Now, assume that \( j^{(V)}_k - i^{(V)}_k = 2 \) and \( V_k \) is not a blocker.  
Hence, \( V_k = a_{i^{(V)}_k+1}a_{i^{(V)}_k} \).  
If \( a_{i^{(V)}_k+1} \in N^{l}_W \), then \( a_{i^{(V)}_k} \not\in N^{r}_W \), and thus \( W^{+} V W^{\ast} \neq W \).  
Similarly, if \( a_{i^{(V)}_k} \in N^{r}_W \), then \( a_{i^{(V)}_k+1} \not\in N^{l}_W \), leading to a contradiction.  
Therefore, there must exist an integer \( 1 \leq k' \leq m_w  \) such that  
\( i^{(W)}_{k'} \leq i^{(V)}_k \leq j^{(V)}_k \leq j^{(W)}_{k'} \).  
In this case, \( V_k \) is covered by \( W_{k'} \).

If \( V_k \) is a blocker, then there must exist an integer \( 1 \leq k' \leq m_w \) such that either \( k' \in I_W \) and \( V_k \) is covered by \( W_{k'} \), or \( k' \in N_W \) and \( i^{(V)}_k = i^{(W)}_{k'} \).

Finally, assume that \( j^{(V)}_k = i^{(V)}_k \). In this case, \( k \in I_V \).  
To satisfy \( W = W^{(I)} N^{l}_W N^{r}_W W^{(I)} \), the second condition must hold.  
\end{proof}

\begin{thm}\label{lltransitive}
The monoid \( LC_n \) is $\ll$-transitive.
\end{thm}

\begin{proof}
Let \( W \), \( V \), and \( U \) be non-empty and non-zero words in \( LC_n \) such that \( W \ll V \ll U \).
Then, by Lemma~\ref{PlusAndStar}, we have \[W=W^{+}V^{+}UV^{\ast}W^{\ast}=W^{(I)} N^{l}_WV^{(I)} N^{l}_VUV^{(I)} N^{r}_VW^{(I)} N^{r}_W.\]

Let \( V_1 V_2 \cdots V_{m_v} \) denote the canonical form of \( V \) over \( A \).

For \( 1 \leq k \leq m_v \), if \( k \in I_V \), then since \( W \ll V \), it follows from Lemma~\ref{llPrim} that
\begin{itemize}
\item If \( i^{(V)}_k < j^{(V)}_k \), then there exists an integer \( k' \in I_W \) such that \( V_k \) is covered by \( W_{k'} \). Then, we have $W^{(I)}V_k=W^{(I)}$.
\item if \( i^{(V)}_k = j^{(V)}_k \), then either there exists an integer \(k'\in I_W\) such that \( V_k \) is covered by \( W_{k'} \), or there exists \( k' \in N_W \) such that \( i^{(V)}_k = i^{(W)}_{k'} \) or \( i^{(V)}_k = i^{(W)}_{k'} + 1 \).

The former case implies \( W^{(I)} a_{i^{(V)}_k} = W^{(I)} \). 

Now, suppose that the latter case holds.  
If \( i^{(V)}_k = i^{(W)}_{k'} \), then we have \( N^{l}_W a_{i^{(V)}_k} = N^{l}_W \).  
Otherwise, \( i^{(V)}_k = i^{(W)}_{k'} + 1 \). Since \( V \ll U \) and \( i^{(V)}_k = j^{(V)}_k \), by Lemma~\ref{llPrim}, \( V_k \) can only cover members of \( I_U \). As \( V_k \) has only one letter, \( V_k \) commute with the segments of $U^{(I)}$. Then, we have \( a_{i^{(V)}_k} U = U a_{i^{(V)}_k} \). Furthermore, as \( a_{i^{(V)}_k} N^{r}_W = N^{r}_W \), it follows that  
\[
W^{+} a_{i^{(V)}_k} U V^{\ast} W^{(I)} N^{r}_W = W^{+} U V^{\ast} W^{(I)} N^{r}_W.
\]
\end{itemize}
Hence, we have $W^{+}V^{(I)} N^{l}_VUV^{\ast}W^{\ast}=W^{+}N^{l}_VUV^{\ast}W^{\ast}$.

Also, if \( k \in N_V \), again by Lemma~\ref{llPrim}, either there exists an integer \( k' \in I_W \) such that \( V_k \) is covered by \( W_{k'} \), or there exists \( k' \in N_W \) such that \( i^{(V)}_k = i^{(W)}_{k'} \). The former case implies \( W^{(I)} a_{i^{(V)}_k} = W^{(I)} \), and the latter case implies \( N^{l}_W a_{i^{(V)}_k} = N^{l}_W \). Consequently, we conclude that \( W^{(I)} N^{l}_W N^{l}_V = W^{(I)} N^{l}_W \).

Hence, we have \( W = W^{+} U V^{\ast} W^{\ast} \). Similarly, we have \( W = W^{+} U W^{\ast} \), and thus \( W \ll U \).

The result follows.
\end{proof}

\begin{lem}\label{ssharpt}
Let \( W \) and \( V \) be non-empty and non-zero words.  
We have \( W \sharp V \neq 0 \) if and only if the following conditions hold:
\begin{enumerate}
    \item For every \( 1 \leq k \leq m_w \):
    \begin{enumerate}
        \item If \( \abs{W_k} > 1 \) and \( k \in I_W \), then there exists an integer \( 1 \leq k' \leq m_v \) such that \( W_k = V_{k'} \).
        \item If \( \abs{W_k} = 1 \), then there exists an integer \( 1 \leq k' \leq m_v \) such that either:
        \begin{itemize}
            \item \( W_k = V_{k'} \), or
            \item \( V_{k'} \) is a blocker and \( W_k = a_{i_{k'}^{(V)}} \).
        \end{itemize}
        \item If \( k \in N_W \), then there exists an integer \( k' \in I_V \) such that \( \abs{V_{k'}} = 1 \) and \( a_{i_k^{(W)} + 1} = a_{i_{k'}^{(V)}} \).
    \end{enumerate}
    \item For every \( 1 \leq k \leq m_v \):
    \begin{enumerate}
        \item If \( \abs{V_k} > 1 \) and \( k \in I_V \), then there exists an integer \( 1 \leq k' \leq m_w \) such that \( V_k = W_{k'} \).
        \item If \( \abs{V_k} = 1 \), then there exists an integer \( 1 \leq k' \leq m_W \) such that either:
        \begin{itemize}
            \item \( V_k = W_{k'} \), or
            \item \( W_{k'} \) is a blocker and \( V_k = a_{i_{k'}^{(W)} + 1} \).
        \end{itemize}
        \item If \( k \in N_V \), then there exists an integer \( k' \in I_W \) such that \( \abs{W_{k'}} = 1 \) and \( a_{i_k^{(V)}} = a_{i_{k'}^{(W)}} \).
    \end{enumerate}
\end{enumerate}
\end{lem}

\begin{proof}  
First, suppose that \( W \sharp V \neq 0 \).  

Let \( 1 \leq k \leq m_W \).  

Since \( W^{\ast} = V^{+} \), we analyze the following cases:  
\begin{enumerate}
    \item If \( \abs{W_k} > 1 \) and \( k \in I_W \), then \( W_k \) must appear in \( V \). Thus, there exists an integer \( 1 \leq k' \leq m_V \) such that \( W_k = V_{k'} \).  
    \item If \( \abs{W_k} = 1 \), then \( W_k \) either matches a corresponding component in \( V \) or aligns with a blocker in \( V \). Specifically, there exists \( 1 \leq k' \leq m_V \) such that either \( W_k = V_{k'} \) or \( V_{k'} \) is a blocker and \( W_k = a_{i_{k'}^{(V)}} \).  
\item If \( k \in N_W \), then there exists \( k' \in I_V \) such that \( \abs{V_{k'}} = 1 \) and \( a_{i_k^{(W)} + 1} = a_{i_{k'}^{(V)}} \), or \( V_{k'} \) is a blocker and \( a_{i_k^{(W)} + 1} = a_{i_{k'}^{(V)}} \).  
If \( V_{k'} \) is a blocker, then as \((WV)^{+} = W^{+}\), there must exist a segment in \( W \) that covers the sequence \( a_{i_k^{(W)}}  a_{i_k^{(W)}+2} \), because \( a_{i_k^{(W)} + 2} = a_{i_{k'}^{(V)}+1} \).  
This contradicts the assumption that \( k \in N_W \).
\end{enumerate}  

Analogously, for every \( 1 \leq k \leq m_V \), the corresponding conditions for \( V \) hold.

Conversely, if the conditions in the lemma hold, it is easy to verify that \( W^{\ast} = V^{+} \), \( (WV)^{+} = W^{+} \), and \( (WV)^{\ast} = V^{\ast} \).
\end{proof}

Suppose that \( W \ll V \). By Lemma~\ref{llPrim}, we identify elements \( k \in N_{W} \) that satisfy at least one of the following conditions and include them in the subset \( {N'}_{(W,V)} \subseteq N_{W} \):
\begin{enumerate}
    \item[(a)] There exists \( k' \in I_{V} \) such that \( i^{(V)}_{k'} = j^{(V)}_{k'} \) and \( i^{(V)}_{k'} = i^{(W)}_k + 1 \).
    \item[(b)] There exists \( k' \in N_{V} \) such that \( i^{(V)}_{k'} = i^{(W)}_k \).
\end{enumerate}

Let \( {N''}_{(W,V)} \) represent the complement of \( {N'}_{(W,V)} \) in \( N_{W} \), defined as  
\[
{N''}_{(W,V)} = N_{W} \setminus {N'}_{(W,V)}.
\]
Define 
\begin{align*}
{N'}^{l}_{(W,V)} = \prod_{k \in {N''}_{(W,V)}} a_{i_{k}}, \quad \text{and}\ {N'}^{r}_{(W,V)} = \prod_{k \in {N'}_{(W,V)}} a_{i_k+1}.
\end{align*}

\begin{lem}\label{NPrimeWandWPrime}
Let \( W'' \ll W' \ll W \).
We have \({N'}_{(W'',W)} \subseteq {N'}_{(W'',W')}\).
\end{lem}

\begin{proof} 
Let \(k \in {N'}_{(W'',W)}\). Then one of the following conditions holds:
\begin{enumerate}
    \item[(a)] There exists \( k' \in I_{W} \) such that \( i^{(W)}_{k'} = j^{(W)}_{k'} \) and \( i^{(W)}_{k'} = i^{(W'')}_k + 1 \).
    \item[(b)] There exists \( k' \in N_{W} \) such that \( i^{(W)}_{k'} = i^{(W'')}_k \).
\end{enumerate}
As \(W' \ll W \) and \( W'' \ll W'\), by Lemma~\ref{llPrim},
we have \(k \in {N'}_{(W'',W')}\).
\end{proof}

In~\cite{Sha-Det2}, \(\ll\)-smoothness is defined as follows:\\
A \(\ll\)-transitive singleton-rich semigroup \(S\) is said to be \(\ll\)-smooth if, for every sequence \(s''\ll s' \ll s\) and \(t''\ll t' \ll t\), the following conditions hold:
\begin{enumerate}
\item 
if \(s''\sharp t''\neq 0\), then we have \(\varphi^{({s''}^{+}s',t'{t''}^{\ast})} = \varphi({s''}^{+}s',t'{t''}^{\ast}).\) 
\item if \(s''\sharp t''\neq 0\), then we have \(({s''}^{+}s')^{\ast}t'{t''}^{\ast} = t''\) if and only if
 \[({s''}^{+}s')^{\ast}t{t''}^{\ast} = t''.\]
\item if \(s''({s''}^{+}s)^{\ast}=s''\), then we have \(s''({s''}^{+}s')^{\ast}=s''\).
\end{enumerate}

We first prove that \( LC_n \) is \( \ll \)-smooth. Then, in the next section, we compute its determinant.

\begin{thm}\label{llsmoothLCn}
The monoid \( LC_n \) is $\ll$-smooth.
\end{thm}

\begin{proof}  
By Corollary~\ref{singletonrich} and Theorem~\ref{lltransitive}, the monoid \( LC_n \) is a $\ll$-transitive and singleton-rich monoid.

Next, we verify that \( LC_n \) satisfies the remaining properties of $\ll$-smooth semigroups to confirm that it is indeed $\ll$-smooth.

Let the sequences  
\( W'' \ll W' \ll W \) and \( V'' \ll V' \ll V \) in \( LC_n \).  
We will now verify these properties one by one as follows:

(1)  
Assume that \( W'' \sharp V'' \neq 0 \).  
By Lemma~\ref{PlusAndStar}, we can write:  
\[
{W''}^{+} W' = {W''}^{(I)} N^{l}_{W''} W'\ \text{and}\ 
V'{V''}^{\ast} = V'{V''}^{(I)} N^{r}_{V''}.
\]

We have \( W'' \ll W' \).
Let
\begin{align*}
{N'}^{lr}_{(W'',W')} = \prod_{k \in {N'}_{(W'',W')}} a_{i_{k}}a_{i_k+1}.
\end{align*}

Then, we can write  
\[
{W''}^{+} W' = {W''}^{(I)} {N'}^{l}_{(W'',W')} {N'}^{lr}_{(W'',W')}\]
and
\[({W''}^{+} W')^{\ast} = {W''}^{(I)} {N'}^{l}_{(W'',W')} {N'}^{r}_{(W'',W')}.
\]


Also, we have \( V'' \ll V'\).

By Lemma~\ref{ssharpt}, we have \({W''}^{(I)} = {W''}^{(I)}N^{l}_{V''}\).
Then, we have \[{W''}^{(I)}V'{V''}^{\ast} = {W''}^{(I)}N^{l}_{V''}V'{V''}^{\ast} = {W''}^{(I)}N^{l}_{V''}V'{V''}^{(I)} N^{r}_{V''}.\]
Thus, by Lemma~\ref{llPrim}, we have 
\[{W''}^{(I)}V'{V''}^{\ast} = {W''}^{(I)}N^{l}_{V''}{V''}^{(I)} N^{r}_{V''}= {W''}^{(I)}{V''}^{(I)} N_{V''}.\]

Hence,
\begin{align*}
({W''}^{+} W')^{\ast}V'{V''}^{\ast} =
{W''}^{(I)} {N'}^{l}_{W''} {N'}^{r}_{W''}{V''}^{(I)}N_{V''}.
\end{align*}
By using twice Lemma~\ref{ssharpt}, we have
\begin{align*}
({W''}^{+} W')^{\ast}V'{V''}^{\ast}
& =
{W''}^{(I)} {N'}^{lr}_{(W'',W')} {N'}^{r}_{(W'',W')}{V''}^{(I)}N_{V''}\\
&=
{W''}^{(I)} {N'}^{lr}_{(W'',W')} {N'}^{r}_{(W'',W')}N_{V''}.
\end{align*}
Hence, again by Lemma~\ref{ssharpt},  we conclude that
\begin{align*}
\varphi({W''}^{+}W',V'{V''}^{\ast}) &= ({W''}^{(I)} {N'}^{lr}_{(W'',W')} {N'}^{r}_{(W'',W')}N_{V''})^{+}\\
&=
{W''}^{(I)} {N'}^{l}_{(W'',W')} {N'}^{r}_{(W'',W')}.
\end{align*}


Now, as 
\begin{align*}
&({W''}^{+}W'\varphi({W''}^{+}W',V'{V''}^{\ast}))^{\ast}=\\
&({W''}^{(I)} {N'}^{l}_{(W'',W')} {N'}^{lr}_{(W'',W')} {W''}^{(I)} {N'}^{l}_{(W'',W')} {N'}^{r}_{(W'',W')})^{\ast}=\\
&({W''}^{(I)} {N'}^{l}_{(W'',W')} {N'}^{lr}_{(W'',W')}  )^{\ast}=\\
&{W''}^{(I)} {N'}^{l}_{(W'',W')} {N'}^{r}_{(W'',W')}=
\varphi({W''}^{+}W',V'{V''}^{\ast})
\end{align*}
and as previously shown,
\begin{align*}
(\varphi({W''}^{+}W',V'{V''}^{\ast})V'{V''}^{\ast})^{+}
&=
({W''}^{(I)} {N'}^{l}_{(W'',W')} {N'}^{r}_{(W'',W')}V'{V''}^{\ast})^{+}\\
&=({W''}^{(I)} {N'}^{lr}_{(W'',W')} {N'}^{r}_{(W'',W')}V'{V''}^{\ast})^{+}\\
&={W''}^{(I)} {N'}^{l}_{(W'',W')} {N'}^{r}_{(W'',W')}\\
&=
\varphi({W''}^{+}W',V'{V''}^{\ast}),
\end{align*}
it follows that \(\varphi^{({W''}^{+}W',V'{V''}^{\ast})} = \varphi({W''}^{+}W',V'{V''}^{\ast}).\) 


(2)
Suppose that \( W'' \sharp V'' \neq 0 \).  

In (1), we established that  
\[
({W''}^{+} W')^{\ast} V {V''}^{\ast} = {W''}^{(I)} {N'}^{lr}_{(W'',W')} {N'}^{r}_{(W'',W')} N_{V''}.
\]  
Since the word \( {W''}^{(I)} {N'}^{lr}_{(W'',W')} {N'}^{r}_{(W'',W')} N_{V''} \) is independent of \( V \), it follows that  
\[
({W''}^{+} W')^{\ast} V {V''}^{\ast} = ({W''}^{+} W')^{\ast} V' {V''}^{\ast},
\]  
and thus (2) holds.


(3) 
Suppose that \(W''({W''}^{+}W)^{\ast}=W''\).


Since \( W'' \ll W \), we have \[
({W''}^{+} W)^{\ast} = {W''}^{(I)} {N'}^{l}_{(W'',W)} {N'}^{r}_{(W'',W)},
\]
and, thus,
\begin{align*}
W''({W''}^{+}W)^{\ast}
&=
W''{W''}^{(I)} {N'}^{l}_{(W'',W)} {N'}^{r}_{(W'',W)}\\
&=
{W''}^{(I)}\prod_{k \in {N}_{W''}} a_{i_{k}}a_{i_{k}+1}{W''}^{(I)} {N'}^{l}_{(W'',W)} {N'}^{r}_{(W'',W)}\\
&=
{W''}^{(I)}\prod_{k \in {N}_{W''}} a_{i_{k}}a_{i_{k}+1}{N'}^{l}_{(W'',W)} 
\end{align*}
As \(W''({W''}^{+}W)^{\ast}=W''\), we get that ${N'}^{l}_{(W'',W)}$ is an empty word and thus ${N''}_{(W'',W)}=\emptyset$.

By Lemma~\ref{NPrimeWandWPrime}, we have \({N'}_{(W'',W)} \subseteq {N'}_{(W'',W')}\) and, thus, ${N''}_{(W'',W')}= \emptyset$.
Hence, \(W''({W''}^{+}W')^{\ast} = W''\).
\end{proof}



\section{Determinant of Layered Catalan monoids}

In this section, we prove that the determinant of the Layered Catalan monoid is non-zero if and only if \( n < 8 \). 

Let \( e \) be an idempotent of \( S \). Define \( X_e \) as the set  
\[
X_e = \{ x_s \in X \mid s \in \widetilde{L}_e \widetilde{R}_e \}.
\]  
Let \( \theta_e(X_e) \) denote the determinant of the submatrix \( \widetilde{L}_e \times \widetilde{R}_e \) extracted from the Cayley table \( (S, \boldsymbol{*}) \).  

In the case where \( S \) is \( \ll \)-smooth, a theorem in~\cite[Theorem 5.11]{Sha-Det2} provides a method to compute the determinant of \( S \).

For $s \in S$, put \[y_s=\sum\limits_{t\ll s}\mu_S(t, s)x_t.\] 
Then, we have
\[
\theta_S(X)
= \pm\prod\limits_{e\in E(S)}\widetilde{\theta}_{e}(Y_e)
\]
where $Y_e=\{y_s\mid s \in \widetilde{L}_e\widetilde{R}_e\}$.
Moreover, the determinant of $S$ is non-zero if and only if $\widetilde{\theta}_{e}(Y_e)\neq 0$, for every idempotent $e$.

In Theorem~\ref{llsmoothLCn}, we prove that \(LC_n\) is \(\ll\)-smooth. Then, we use Theorem 5.11 in~\cite{Sha-Det2} to compute its determinant.

\begin{lem}\label{NumberOfPlusAndStar}
Let \( e \) be a non-zero and non-empty idempotent with the canonical form  
\( W_1 W_2 \cdots W_{m_e} \) over \( A \).  

Define \( l_1 \) and \( l_2 \) as follows:
\begin{align*}
    l_1 &= \abs{\big\{ W_i \mid \abs{W_i} = 1, \ \text{and}\ d(W_{i-1}, W_i) \geq 3 \big\}}, \\
    l_2 &= \abs{\big\{ W_i \mid \abs{W_i} = 1, \ \text{and}\ d(W_i, W_{i+1}) \geq 3 \big\}}.
\end{align*}
where \( d(W_a, W_b) \) denotes the number of letters in the distance between \( W_a \) and \( W_b \), given by \( d(W_a, W_b) = i_b - j_a + 1 \).
If \( e \) consists of only one segment, we assume that distances are measured cyclically within \( e \).  

Then, the following holds:
\[
\abs{\widetilde{L}_e} = 2^{l_1}, \quad \text{and} \quad \abs{\widetilde{R}_e} = 2^{l_2}.
\]
\end{lem}

\begin{proof}
Let \( e_1 \) be the set of segments \( W_i \) in \( e \) such that \( \abs{W_i} = 1 \) and  
\[
d(W_{i-1}, W_i) \geq 3.
\]  
Define \( e_2 \) as the complement of \( e_1 \) within the set of all segments of \( e \), i.e.,  
\[
e_2 = \{ W_i \mid W_i \notin e_1 \}.
\]  

Now, suppose that \( W \in \widetilde{L}_e \).  
By Lemma~\ref{PlusAndStar}, we have  
\(
W^{\ast} = W^{(I)} N^{r}_W = e.
\)  
This implies that every segment \( W_i \) in \( e_2 \) must appear in \( W^{(I)} \).  
On the other hand, each segment in \( e_1 \) can either be included in \( W^{(I)} \) or in \( N^{r}_W \).  
Then, for every \( W \in \widetilde{L}_e \), a segment in \( e_1 \) may or may not act as a blocker in \( W \).  
It follows that the number of possible choices for \( \widetilde{L}_e \) is given by  
\(
\abs{\widetilde{L}_e} = 2^{l_1}.
\)  

Similarly, we obtain  
\(
\abs{\widetilde{R}_e} = 2^{l_2}.
\)  

\end{proof}

\begin{thm}\label{Det-LCn}
We have \( \theta_{LC_n}(X) \neq 0\) if and only if \(n< 8\).
\end{thm}

\begin{proof}
By Lemma~\ref{PlusAndStar}, we have  
\[
\widetilde{L}_{a_1a_5} = \{a_1a_5, a_1a_4a_5, a_na_1a_5, a_na_1a_4a_5\}
\]
and  
\[
\widetilde{R}_{a_1a_5} = \{a_1a_5, a_1a_2a_5, a_1a_5a_6, a_1a_2a_5a_6\}.
\]

For \( n \geq 8 \), we have \( \abs{\widetilde{L}_{a_1a_5}} = \abs{\widetilde{R}_{a_1a_5}} = 4 \).  
By Lemma 5.8 in \cite{Sha-Det2}, we compute the submatrix \( \widetilde{L}_{a_1a_5} \times \widetilde{R}_{a_1a_5} \), extracted from the Cayley table \( (S, \boldsymbol{*}) \), as follows:
\[
\begin{array}{c|cccc}
 & a_1a_5 & a_1a_2a_5 & a_1a_5a_6 & a_1a_2a_5a_6 \\ \hline
 a_1a_5 & a_1a_5 & a_1a_2a_5 & a_1a_5a_6 & a_1a_2a_5a_6 \\
 a_1a_4a_5 & a_1a_4a_5 & \cdot & \cdot & \cdot \\
 a_5a_na_1 & a_5a_na_1 & \cdot & a_na_1a_5a_6 & \cdot \\
 a_4a_5a_na_1 & a_4a_5a_na_1 & \cdot & \cdot & \cdot 
\end{array}
\]
where the zero is represented with a dot.
To compute the entries the table, at position \((s,t)\), we only need to verify that  
\[
(st)^{+} = s^{+} \quad \text{and} \quad (st)^{\ast} = t^{\ast}.
\]

Since its determinant is zero, 
this is sufficient to show that the determinant of \( \theta_{LC_n}(X) \) is zero, for every \( n \geq 8 \). 

It is evident that the determinant of the monoid \( LC_1 \) is non-zero.  
The Cayley table of \( LC_2 \) is as follows:
\[
\begin{array}{c|ccc}
     & a_1   & a_2   & 1 \\ \hline
 a_1 & a_1   & \cdot & a_1 \\
 a_2 & \cdot & a_2   & a_2 \\
 1   & a_1   & a_2   & 1
\end{array}
\]
which indicates that the determinant is non-zero.

Now, suppose that $3 \leq n \leq 7$.

Let $e$ be an idempotent of $LC_n$
and let \( W_1 W_2 \cdots W_{m_e} \) denote the canonical form of \( e \) over \( A \).
As \( n \leq 7 \) and there is a two-letter distance between every neighboring segments of \( W \), we have \( m_e \leq 2 \).

By Lemma~\ref{NumberOfPlusAndStar}, since \( n \leq 7 \), we have  
\[
l_1=\abs{\big\{ W_i \mid \abs{W_i} = 1, \ \text{and} \ d(W_{i-1}, W_i) \geq 3 \big\}} \leq 1.
\]  
Thus, it follows that  
\(
1 \leq \abs{\widetilde{L}_e} \leq 2.
\)

Then, we have the following cases:

\( \abs{\widetilde{L}_e} = 1 \) and \( m_e = 1 \):

In this case, \( e \) is an idempotent with a single segment of length greater than one.  
Thus, the submatrix \( \widetilde{L}_e \times \widetilde{R}_e \) contains only a single entity, which is \( e \), and its determinant is non-zero.

\( \abs{\widetilde{L}_e} = 1 \) and \( m_e = 2 \):

This case occurs only when \( n = 7 \).  
It implies that \( e = a_i a_{i+4} a_{i+3} \) for some \( i \) modulo \( 7 \).  
Again, we obtain the same result as in the previous case.

\( \abs{\widetilde{L}_e} = 2 \) and \( m_e = 1 \):

In this case, \( e \) is a single letter, namely \( a_i \), for some integer \( i \).  
By Lemma~\ref{NumberOfPlusAndStar}, we have \( \abs{\widetilde{R}_e} = 2 \).  
The submatrix \( \widetilde{L}_{a_i} \times \widetilde{R}_{a_i} \) is given by:
\[
\begin{array}{c|cc}
            & a_i    & a_i a_{i+1} \\ \hline
 a_i        & a_i    & a_i a_{i+1} \\
 a_{i-1}a_i & a_{i-1}a_i & \cdot 
\end{array}
\]
That its determinant is non-zero.

\( \abs{\widetilde{L}_e} = 2 \) and \( m_e = 2 \):
 
This case occurs only when \( n = 7 \).  
Here, \( e \) consists of two single-letter segments separated by a distance of 3 letters, namely \( a_i a_{i+4} \), for some integer \( i \).  
By Lemma~\ref{NumberOfPlusAndStar}, we also have \( \abs{\widetilde{R}_e} = 2 \).  
The submatrix \( \widetilde{L}_{a_i a_{i+4}} \times \widetilde{R}_{a_i a_{i+4}} \), extracted from the Cayley table \( (S, \boldsymbol{*}) \), is given by:
\[
\begin{array}{c|cc}
                      & a_i a_{i+4}         & a_i a_{i+1} a_{i+4} \\ \hline
 a_i a_{i+4}          & a_i a_{i+4}         & a_i a_{i+1} a_{i+4} \\
 a_i a_{i+3} a_{i+4}  & a_i a_{i+3} a_{i+4} & \cdot 
\end{array}
\]
Since its determinant is non-zero, it follows that \( \det(LC_n) \neq 0 \) for \( 1 \leq n \leq 7 \).
\end{proof}


\section*{Acknowledgments}
The author was partially supported by CMUP, member of LASI, which is financed by national funds through FCT -- Funda\c c\~ao para a Ci\^encia e a Tecnologia, I.P., under the projects with reference UIDB/00144/2020 and UIDP/00144/2020.


\bibliographystyle{plain}
\bibliography{ref-Det}

\end{document}